\documentclass[letterpaper, 10pt, conference]{cssconf}  
\IEEEoverridecommandlockouts
\usepackage{epsfig}
\usepackage{amsmath}
\usepackage{nomencl}
\usepackage{subfig}
\usepackage{algo}
\DeclareMathOperator*{\argmin}{arg\,min}

\newtheorem{lemma}{Lemma}

\makenomenclature

\title{\LARGE \bf Pursuit on a Graph Using Partial Information}
\author{K. Krishnamoorthy, D. Casbeer, P. Chandler, and M. Pachter  
\thanks{Corresponding author: K. Krishnamoorthy {\tt\small krishnak@ucla.edu}}
\thanks{This work is approved for public release, distribution unlimited: 88ABW-2014-4329}
\thanks{K. Krishnamoorthy is with the InfoSciTex corporation, Dayton, OH 45431}
\thanks{D. Casbeer is with the Autonomous Control Branch, Air Force Research Laboratory, Wright-Patterson AFB, OH 45433}
\thanks{P. Chandler (Retd.) was with the Autonomous Control Branch, Air Force Research Laboratory, Wright-Patterson AFB, OH 45433}
\thanks{M. Pachter is with the Department of Electrical Engineering, Air Force Institute of Technology, Wright-Patterson AFB, OH 45433}
}
\begin{document}

\maketitle
\thispagestyle{empty}
\pagestyle{empty}

\begin{abstract}
The optimal control of a ``blind"  pursuer searching for an evader moving on a road network and heading at a known speed toward a set of goal vertices is considered. To aid the ``blind" pursuer, certain roads in the network have been instrumented with Unattended Ground Sensors (UGSs) that detect the evader's passage.
When the pursuer arrives at an instrumented  node, the UGS therein informs the pursuer if and when the evader visited the node. The pursuer's motion is not restricted to the road network. In addition, the pursuer can choose to wait/loiter for an arbitrary time at any UGS location/node. At time $0$, the evader passes by an entry node on his way towards one of the exit nodes. The pursuer also arrives at this entry node after some delay and is thus informed about the presence of the intruder/evader in the network, whereupon the chase is on - the pursuer is tasked with capturing the evader.
Because the pursuer is ``blind", capture entails the pursuer and evader being collocated at an UGS location. If this happens, the UGS is triggered and this information is instantaneously relayed to the pursuer, thereby enabling capture.
On the other hand, if the evader reaches one of the exit nodes without being captured, he is deemed to have escaped. We provide an algorithm that computes the maximum initial delay at the entry node for which capture is guaranteed. The algorithm also returns the corresponding optimal pursuit policy.
\end{abstract}
\printnomenclature

\nomenclature{$P_k$}{Evader path indexed by $k=1,\ldots,n$ }%
\nomenclature{$|P_k|$}{Length of path $k$ }%
\nomenclature{$\ell_k$}{number of UGSs along path $k$ }%
\nomenclature{$\mathcal P_j$}{Set of path indices that contain UGS $j$ }%
\nomenclature{$j$}{UGS index taking values $1,\ldots,m$ }%
\nomenclature{$d_V(i,j)$}{Pursuer travel time from node $i$ to $j$ }%
\nomenclature{$\mathcal T_k(i)$}{Evader arrival time at $i^{th}$ UGS along path $k$}
\nomenclature{$\mathcal L_j(k)$}{Evader time of visit to UGS $j$ along path $k$ }
\nomenclature{$\mathcal I$}{Evader path uncertainty set}
\nomenclature{$p$}{Pursuer UGS location index}
\nomenclature{$u$}{Pursuer control decision }
\nomenclature{$y$}{Pursuer observation at UGS}

\section{Introduction}
\label{sec:intro}
We are concerned with capturing a ground target moving on a road network. The operational scenario is as follows. The access road network to a restricted (protected) zone is instrumented with Unattended Ground Sensors (UGSs), placed at critical locations. As the target, referred to as the ``evader", passes by an UGS, the UGS is triggered. A triggered UGS turns, say, from \emph{green} to \emph{red} and records the evader's time of passage.
The UGSs are placed on certain edges of the graph.
We assume that the speed of the evader, the layout of the road network and the placement of the UGSs is known to the pursuer. When the pursuer arrives at an UGS location, the information stored by the UGS is uploaded to the pursuer, namely, the green/red status of the UGS and, if the UGS is red, the time elapsed (delay) since the evader's passage.
The evader can be captured in one of two ways: either the evader and pursuer synchronously arrive
at an UGS location, or the pursuer is already loitering/waiting at an UGS location when the evader
arrives there. In both cases, the UGS is triggered, instantaneously informs the pursuer, and the evader is captured. 
The decision problem for the pursuer is to select which UGS to visit next, including possibly staying at the current UGS location (and if so, for how long?) awaiting the arrival of the evader. The decisions are made by the pursuer at discrete time instants, immediately after arriving at and interrogating an UGS. Without loss of generality, we assume the evader is traveling on the road network at unit speed. The pursuer, on the other hand, is not restricted to be on the road network, although only upon visiting an UGS can he update the information state. 
In addition, the pursuer can also wait/loiter at an UGS location for an arbitrary amount of time.
\begin{figure}[h]
\begin{center}
\includegraphics[width=\linewidth]{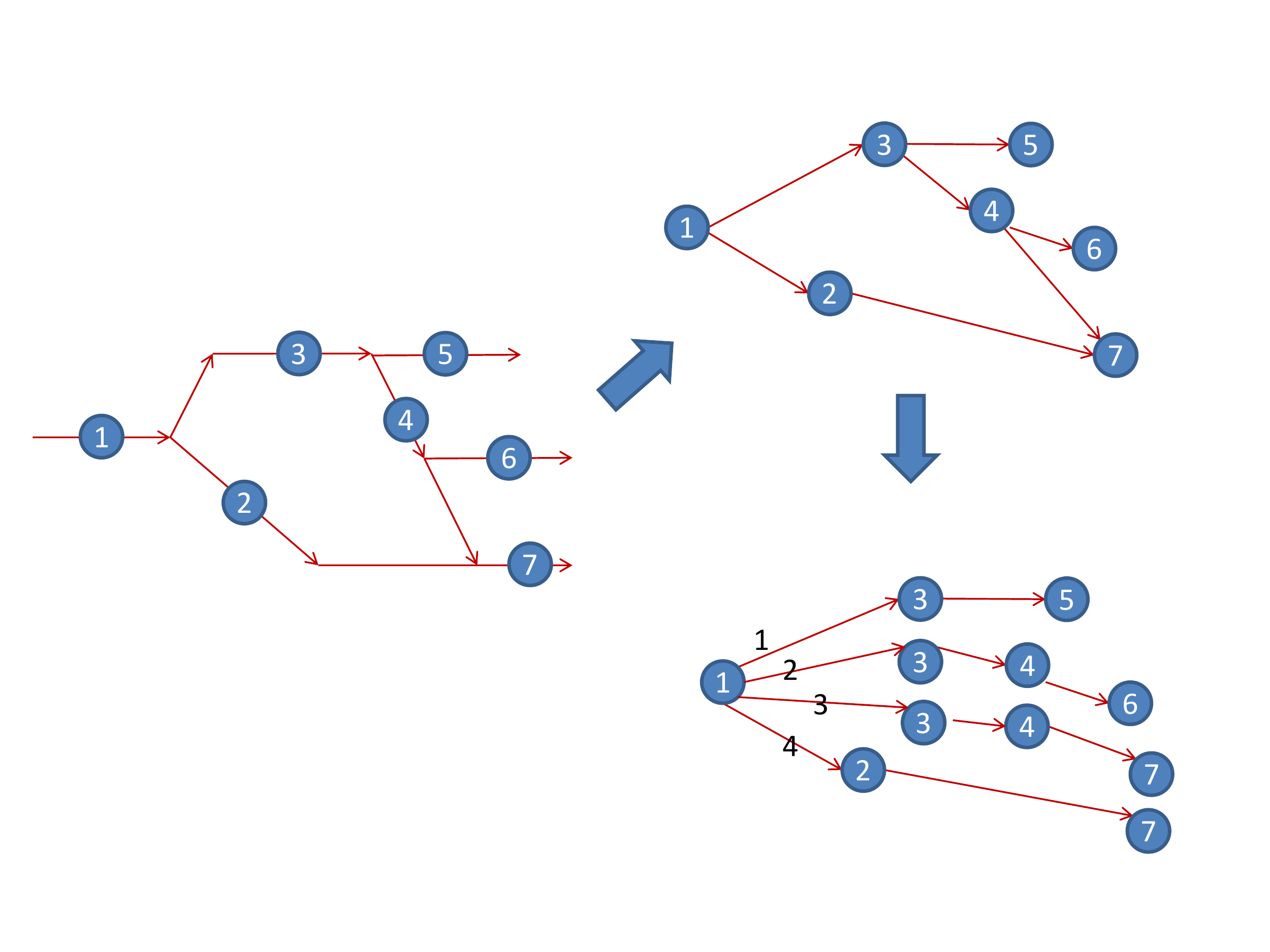}
\caption{Road Network, UGSs Graph, and Four Possible Evader Paths} \label{fig:graph}
\end{center}
\end{figure}

In Fig.~\ref{fig:graph} an illustrative road network is shown. 
The roads are shown in red (arrows indicate direction of travel) and the numbered UGSs are blue circles. Let there be $m$ UGSs on the network, indexed by $j= 1,\ldots,m$. %
Since information is only available (and capture only possible) at the UGS locations, we focus on the embedded graph, $G(\mathcal U,E)$, that has the UGS locations as vertices, i.e., $\mathcal U=\{1,\ldots,m\}$. We make the critical assumption that $G$ is a directed acyclic graph.
To visualize the setup, see Fig.~\ref{fig:graph}, where the corresponding graph, $G$, is shown in the top right. Here, node $1$ is the entry node into the network. A directed edge, $e\in E$ between two nodes on the graph has a weight that equals the distance along the road network between the nodes.
For each $j\in\mathcal U$, let $C(j)\subset\mathcal U$ indicate the set of child nodes that the evader can get to from $j$. Let $\mathcal G=\{j:j\in\mathcal U \mbox{ and } C(j)=\emptyset\}$ indicate the set of exit/goal nodes that the evader is heading towards. 
In Fig.~\ref{fig:graph}, nodes $5$, $6$ and $7$ are the exit nodes.
Furthermore, for each $c\in C(j)$, let the distance along the road network between the parent and child node be indicated by $T(j,c)$. Since the evader travels at unit speed, this is also the time taken by the evader to go from node $j$ to its child node $c$. The pursuer's travel time from node $i$ to node $j$ is given by a scaled distance metric, $d_V(i,j)$. For example, it could represent the Euclidean distance between the nodes divided by the pursuer's speed. Here, we allow the metric to be more general, so long as it satisfies the triangle inequality, i.e.,
\begin{equation}
\label{eq:triangle}
d_V(i,j) \leq d_V(i,s) + d_V(s,j),
\end{equation}
for any $i,j,s\in\mathcal U$ and $d_V(j,j)=0$, $\forall j\in\mathcal U$. The above generalization allows us to model different scenarios, e.g., the pursuer could be an Unmanned Air Vehicle (UAV). We assume that the pursuer is faster than the evader, that is, the pursuer's travel time between any UGS and its child node is strictly less than the evader's travel time between the two nodes, i.e., 
\begin{equation}
\label{eq:spdAdv}
d_V(j,c) < T(j,c),\;\forall c\in C(j),\;\forall j\in\mathcal U.
\end{equation}
Without loss of generality, we assume that the evader (upon entering the network ) first visits node $1$ at time $0$ and $1\notin\mathcal G$.
Let there be $n\ (\geq 1)$ possible evader paths denoted by $P_1,\ldots,P_{n}$ emanating from node $1$ and terminating at an exit node. For the example problem, see the enumeration of the $4$ possible evader paths shown on the bottom right of Fig.~\ref{fig:graph}. 
We represent an evader path $P_k$, $1\leq k\leq n$, by the following notation: $P_k=(1\rightarrow s_k^2\rightarrow \ldots \rightarrow s_k^{\ell_k})$, where $s_k^{i}$ is the $i^{th}$ UGS along path $k$ and $s_k^{\ell_k}\in\mathcal G$. 
Here, ${\ell_k}$ is the number of UGSs along path $k$. For example, in Fig. 1 $P_1=(1\rightarrow 3\rightarrow 5)$, so $s_1^2 =3$, $s_1^3 =5$ and ${\ell_1}=3$.

\subsection{Properties of the Evader's Path}
\label{sec:pathprop}
Let $\mathcal T_k(j)$ be the time of arrival of the evader to the $j^{th}$ UGS along path $k$.
\begin{equation}
\label{eq:time2Ugs}
\mathcal T_k(j)  = T(1,s_k^2) + \sum_{r=2}^{j-1}T(s_k^r,s_k^{r+1}),\;j=2,\ldots,\ell_k.
\end{equation}
So, the length of each path is given by $|P_k|=\mathcal T_k({\ell_k})$. If the evader were to pick the shortest path to an exit node, then he would choose, $\bar{k}=\argmin_{k=1}^n|P_k|$.
Since $G$ is a directed acyclic graph, the evader cannot visit any particular UGS more than once. However, it is possible that the evader can reach an UGS, $U \in\mathcal U$ via different paths. So, for
UGS/node $U_j$ in the graph, $j=1,\ldots,m$, we associate the set, $\mathcal L_j=\{\mathcal L_j(1),\ldots,\mathcal L_j(n)\}$, where $\mathcal L_j(k)$ is the time at which the evader would visit node $j$ while traveling along path $k$. Here, time is measured relative to time $0$, when the evader visits node $1$. If node $j$ does not appear in some path $k\in\{1,\ldots,n\}$, then we set the corresponding time, $\mathcal L_j(k)=\infty$. We assume without loss of generality, that $\forall j,\exists k$ such that $\mathcal L_j(k)<\infty$. This condition implies that every UGS appears, at least, in one of the paths. Clearly, if this were not the case, such an  UGS can be removed from consideration. By definition, we have $\mathcal L_1=\{0,\ldots,0\}$, since node $1$ is visited by the evader at time $0$ and along every possible evader path emanating from UGS 1. We also define the set, $\mathcal P_j$, $j=1,\ldots,m$, to be the set of paths that contain node $j$. By definition, $\mathcal P_j =\{k:\mathcal L_j(k)<\infty,\;k=1,\ldots,n\}$ and $\mathcal P_j\neq \emptyset,\;\forall j$.
%
We define the initial uncertainty in evader path information available to the purser to be $\mathcal I_0$.
Since the evader could have taken any one of $n$ paths, $\mathcal I_0=\{1,\ldots,n\}$. Note that this definition of evader position uncertainty  appears to be unusual in that for small initial delays, the pursuer will know where the evader is on the road that contains node $1$ (e.g., see left plot in Fig.~\ref{fig:graph}) and so, there is no uncertainty in his position/state; but we still say his path is uncertain in that $\mathcal I _0=\{1,2,3,4\}$ - see bottom right plot in Fig.~\ref{fig:graph}. This is an important point: because the tacitly assumed  information pattern is s.t. the evader has no situational awareness, one could argue that the evader might as well decide on his ``strategy", namely, what path he will take, at $t=0$ - in other words, the evader operates in $open$-$loop$. So we stipulate that at each point in time, and based on the evidence collected so far, the information of the pursuer is the currently feasible set of possible paths, one of which the evader, having made his choice at time $0$, is currently traveling on. This definition of path uncertainty, meaning, the uncertainty about which of the $n$ paths the evader is actually traveling on, results in a significant simplification of the underlying coupled estimation and control problem. Hereafter, we shall use the words uncertainty and information interchangeably with reference to the set of complete paths that the evader is possibly traveling on.

\subsection{Evolution of System State}
\label{sec:StateEval}
Even though the pursuer and evader motion evolve in continuous time, decisions are made (by the pursuer only) at discrete time steps. The pursuer makes these decisions immediately after reaching an UGS location at time $t$ and obtaining the measurement $y$ therein: $y=-1$ for ``green", or $y=d$ for ``red" + delay $d$. 
Let the pursuer position at decision time $t$ be specified by the UGS index, $p\in\{1,\ldots,m\}$. 
The decision variable, $u$ indicates the UGS location $u\in\{1,\ldots,m\}$ that the pursuer should visit next. 

The control action $u$ is dependent on the current time, pursuer position and most recent information state: $u=\mathcal F(t,\mathcal I,p)$, where  the mapping $\mathcal F$ is to be determined by an optimality principle - see \eqref{eq:1stepTime} in the sequel.
So, the pursuer's position and pursuer decision time evolve according to:
\begin{eqnarray}
\label{eq:decTime}
p^+ &=& u,\nonumber \\
t^+ &=& \left\{
\begin{array}{l}t + d_V({p},{u}),\;u\neq p\\
        \min_{k\in \mathcal I} \mathcal L_p(k),\;u=p
        \end{array}\right.
\end{eqnarray}
So, if the pursuer decides to stay put at the current location, the next decision epoch is the earliest possible time at which new information becomes available at the current UGS $p$. We denote by $y$ the measurement the pursuer made at node ${p}$. The observation could either be a red UGS $p$ with delay $d\geq 0$ i.e., $y=d$, or a green UGS $p$; whereupon the observation is denoted by $y=-1$. Note that the pursuer may choose $u=p$ only if the observation $y=-1$. If the pursuer observes a red UGS, it confirms that the evader did pass through UGS $p$ and there is no value in the pursuer staying at $p$ any longer. Indeed, it would be detrimental to the search effort (in terms of time to capture).

Suppose the evader path uncertainty information available to the pursuer at $p$ is  $\mathcal I$.
We calculate the information/path uncertainty set at time $t^+$ for the two possible observations at $u$ as follows:
\\
\textbf{Red} (${ y^+=d\geq 0}$):
The pursuer will observe a red UGS with delay $d\geq0$ where $d\in\{s|s=t^+-\mathcal L_u(k),s\geq 0,k\in\mathcal P_u\cap \mathcal I\}$. This implies that the evader was at the location of UGS $u$ at time $t^+-d$. Therefore, the information at time $t^+$ will be:
\begin{equation}
\label{eq:redUpdate}
\mathcal I^+ (u,d)= \{k: k\in\mathcal I,\mathcal L_{u}(k)=t^+-d\}.
\end{equation}
So we only retain those paths from $\mathcal I$ that are consistent with the evader passing through ${u}$ at time $t^+-d$.
\\
\textbf{Green} (${ y^+= -1}$): The pursuer will observe a green UGS at time $t^+$. This implies that the evader has not visited ${u}$ thus far. Therefore, the information update is given by:
\begin{equation}
\label{eq:greenUpdate}
\mathcal I^+ (u,-1)= \{k: k\in\mathcal I,\mathcal L_{u}(k)>t^+\}.
\end{equation}
So we only retain those paths from $\mathcal I$ that are consistent with the evader passing through ${u}$ at a time greater than $t^+$.
The game will terminate at UGS $p^+$ if at time $t^+$ the new observation is $y=0$. It is also possible that having periodically updated the path uncertainty set $ \mathcal I$ and reapplied \eqref{eq:decTime}, the pursuer stayed put at UGS $p$ until time $\max_{k\in \mathcal I}\mathcal L_p (k)$ whereupon if the last observation $y=0$ the evader is captured. If this observation is $y=-1$ instead, implying that the evader did not take any of the paths that pass through $p$, the control $u\neq p$ is applied and the pursuer finally moves on. The crucial point here is that although ``to wait or not" is a decision to be made by the pursuer, the waiting time itself is \emph{purely} determined by the evader arrival times and pursuer observations. This comes about because of the assumptions: 1) constant evader speed and 2) acyclic graph.

\section{Optimization Problem Statement}
\label{sec:probstat}
The evader passes by node $1$ at time $0$. The pursuer arrives for the $1^{st}$ time at node $1$ at time $t_0>0$ and is tasked with capturing the evader. Obviously, (see Fig.~\ref{fig:graph}) when $t_0$ is small capture is possible, given the pursuer's speed advantage \eqref{eq:spdAdv}. On the other hand, if $t_0$ is large, the evader will likely escape, no matter what the pursuer does. We are interested in computing the maximum initial delay $t_0$ for which a capture guarantee exists. This is valuable information in an operational scenario, for the following reason. The road network could lead to a protected area, that is being guarded against (ground) intrusions by security forces and the pursuer could be an UAV. In this case, it would be advantageous to know what is the maximum delay for which a capture guarantee exists. If the actual initial delay measured by the UAV exceeds the maximum, a security alert ``close the gates!" could be issued and additional resources allocated to intercept the threat. On the other hand, if the actual delay encountered is no greater than the maximum, then the UAV can autonomously pursue the evader, isolate it and transmit the captured image to a human operator for further action.

To pose this as an optimization problem,
we introduce the following concept.
Let $\mathcal D(1|\mathcal I_0)>0$ be the \emph{latest} time that the pursuer can arrive at/leave node $1$ and still capture the evader, knowing that the evader could have taken any one of $n$ paths, $P_1,\ldots,P_n$. Again, time is measured relative to time $0$ which is the time the evader passes node $1$. The evader path information available to the pursuer at node $1$ is given by $\mathcal I_0=\{1,\ldots,n\}$. In a similar fashion, for any UGS, $j=1,\ldots,m$, we define $\mathcal D(j|\mathcal I)$ to be the latest time the pursuer can arrive at/leave node $j$ and guarantee capture, armed with the path information $\mathcal I$. Note that the arrival time to an UGS $=$ the departure time, also in the case where the UGS is ``green" and the pursuer decides to stay put. If the pursuer arrives at node $j$ at time $t>0$ and $t\leq \mathcal D(j|\mathcal I)$, let $\mu(j|\mathcal I)\in\{1,\ldots,m\}$ be the corresponding UGS index to which the pursuer should head towards next, to enable capture. 

Recall that each path $P_k$, $k=1,...,n$, contains an exit node and the exit node of path $k$ is $s_k^{\ell_k}$. For the exit node $s_k^{\ell_k}$, the latest time that the pursuer can arrive there and still guarantee capture, knowing that the evader has taken path $P_k$ is clearly $|P_k|$, the time at which the evader reaches the said node. Thus,
\begin{equation}
\label{eq:bndMaxDel}
\mathcal D( s_k^{\ell_k}|\{k\}) = |P_k|.
\end{equation}
Concerning the pursuer's strategy $\mu$: if $t<|P_k|$, $\mathcal \mu( s_k^{\ell_k}|\{k\})=  s_k^{\ell_k}$ i.e., the pursuer stays put at the exit node.
In general, if the path information is the singleton $\{k\}$, the corresponding latest pursuer arrival time for node $j$, $j=1,\ldots,m,$ is given by:
\begin{eqnarray}
\label{eq:bndMaxDelAll}
\mathcal D(j| \{k\}) &=&\max_{i=1}^{\ell_k}\left[\mathcal T_k({i})-d_V(j,s_k^i)\right],\nonumber \\
&=& |P_k|-d_V(j,s_k^{\ell_k}), \;\forall k.
\end{eqnarray}
The second equality above follows from the triangle inequality \eqref{eq:triangle} and speed advantage \eqref{eq:spdAdv} assumptions. In essence, the pursuer reaches the exit node $ s_k^{\ell_k}$ of path $k$ from node $j$, just in time to capture the evader. So, the corresponding ``go to" UGS is given by, $u=\mu (j| \{k\}) =s_k^{\ell_k}$.
\begin{lemma}
If the path uncertainty set $\mathcal I$ satisfies $\mathcal I\subseteq \mathcal{P}_j$ for some $j\in\{1,...,m\}$, then:
\label{lem:lowbnd}
\begin{eqnarray}
\label{eq:bndIntUgs}
\mathcal D(j| \mathcal I) &\geq& \min_{k\in\mathcal I}\mathcal L_j(k).
\end{eqnarray}
\end{lemma}
\begin{proof}
Since $\mathcal I\subseteq \mathcal{P}_j$, all the paths in the uncertainty set $\mathcal I$ go through node $j$. So, the pursuer can guarantee capture by
arriving at node $j$ at time $t=\min_{k\in\mathcal I}\mathcal L_j(k)$, which is the earliest time that the evader can pass through node $j$ by taking any path, $k\in\mathcal I$.
\end{proof}
\subsection{Max-Min Optimization}
\label{sec:DP}
Suppose the pursuer is at UGS index $p$ with path information $\mathcal I$ and decides to visit $u$ next.
Upon reaching $u$, the information will change to: $\mathcal I^+(u,y)$, where $y$ is the observation that the pursuer will make at $u$. Recall that $\mathcal I^+(u,y)$ is updated according to \eqref{eq:redUpdate} and \eqref{eq:greenUpdate} for the red and green UGS observations respectively. By definition, $\mathcal D({u}|\mathcal I^+(u,y))$ is the latest time at which, armed with the new information $\mathcal I^+(u,y)$, the pursuer can arrive at/leave $u$ and still guarantee capture of the evader. So, the latest time that the pursuer can leave $p$ and still capture the evader should satisfy the Recursive Equation (RE):
\begin{equation}
\label{eq:1stepTime}
\mathcal D({p}|\mathcal I) = \max_{u\in\mathcal U} \left[\min_{y\geq -1}\mathcal D({u}|\mathcal I^+(u,y))-d_V(p,u)\right].
\end{equation}
This is so, because before visiting $u$ the pursuer cannot know whether the observation will be a red or green UGS. Hence, to guarantee capture, it has to assume the worst-case scenario that will result in the smaller of two possible pursuer exit times at $u$. To compute the latest pursuer exit time from $p$, we subtract the travel time from $p$ to $u$. Finally, we take the $\max$ over all possible nodes to get the latest possible exit time from $p$ with a capture guarantee. Per our convention, the corresponding optimal control, $\mu({p}|\mathcal I)=u^*$, where $u^*$ is the maximizing control in \eqref{eq:1stepTime}.
We will use RE \eqref{eq:1stepTime} to compute $\mathcal D({1}|\mathcal I_0)$. Before doing so, we introduce a control constraint, $u\in\mathcal B({\mathcal I})\subset\mathcal U$ in \eqref{eq:1stepTime} that will enable us to compute $\mathcal D({1}|\mathcal I_0)$ in an orderly recursive fashion.  In the next section, we will show that this constraint does not result in any loss in optimality.

\section{Ordered Recursive Solution}
\label{sec:DPsol}
Consider the simplest possible scenario: $n=1$ i.e., there is only one path from the start node $1$ to some exit node, $s_1^{\ell_1}\in\mathcal G$. To guarantee capture, it is sufficient for the pursuer to get to $s_1^{\ell_1}$ no later than the time that the evader gets there. So, the maximum delay at node $1$ with a capture guarantee  is given by,
\begin{equation}
\label{eq:1path}
\mathcal D(1|\{1\}) = |P_1|-d_V(1,s_1^{\ell_1})>0,
\end{equation}
where the inequality follows from the pursuer speed advantage assumption \eqref{eq:spdAdv}. The optimal policy dictates, $\mu(1)=s_1^{\ell_1}$. For $n=1$, there is no uncertainty in the evader's path and so, the pursuer heads straight to the exit node $s_1^{\ell_1}$ and ``captures" the evader. 
This scenario is also reflected in \eqref{eq:bndMaxDelAll}, where the evader's path $k$ is known to the pursuer. Since we are interested in the case where there is uncertainty in the path, our only recourse is to the RE \eqref{eq:1stepTime}, as applied to node $1$ under information, $\mathcal I_0=\{1,\ldots,n\}$, and so,
\begin{equation}
\label{eq:1stepTime0}
\mathcal D(1|\mathcal I_0) = \max_{u} \left[\min_{y\geq -1}\mathcal D({u}|\mathcal I^+(u,y))-d_V(1,u)\right]
\end{equation}
As mentioned earlier, the above equation is recursive in nature. The only exception is the case where the uncertainty set's cardinality is $1$, whereupon \eqref{eq:bndMaxDelAll} provides us the values of:
\begin{equation}
\label{eq:unSetCard1}
\mathcal D(j|\{k\}),\; j=1,\ldots,m,\;k=1,\ldots,n.
\end{equation}
A natural question that arises is the following: could we compute the exit times corresponding to uncertainty sets of cardinality $2$ given \eqref{eq:unSetCard1}? The answer is yes and
\begin{figure}[h]
\centering
\begin{tabular}{c}
\subfloat[][Road Network on a Grid with Coordinates]{\includegraphics[width=\linewidth]{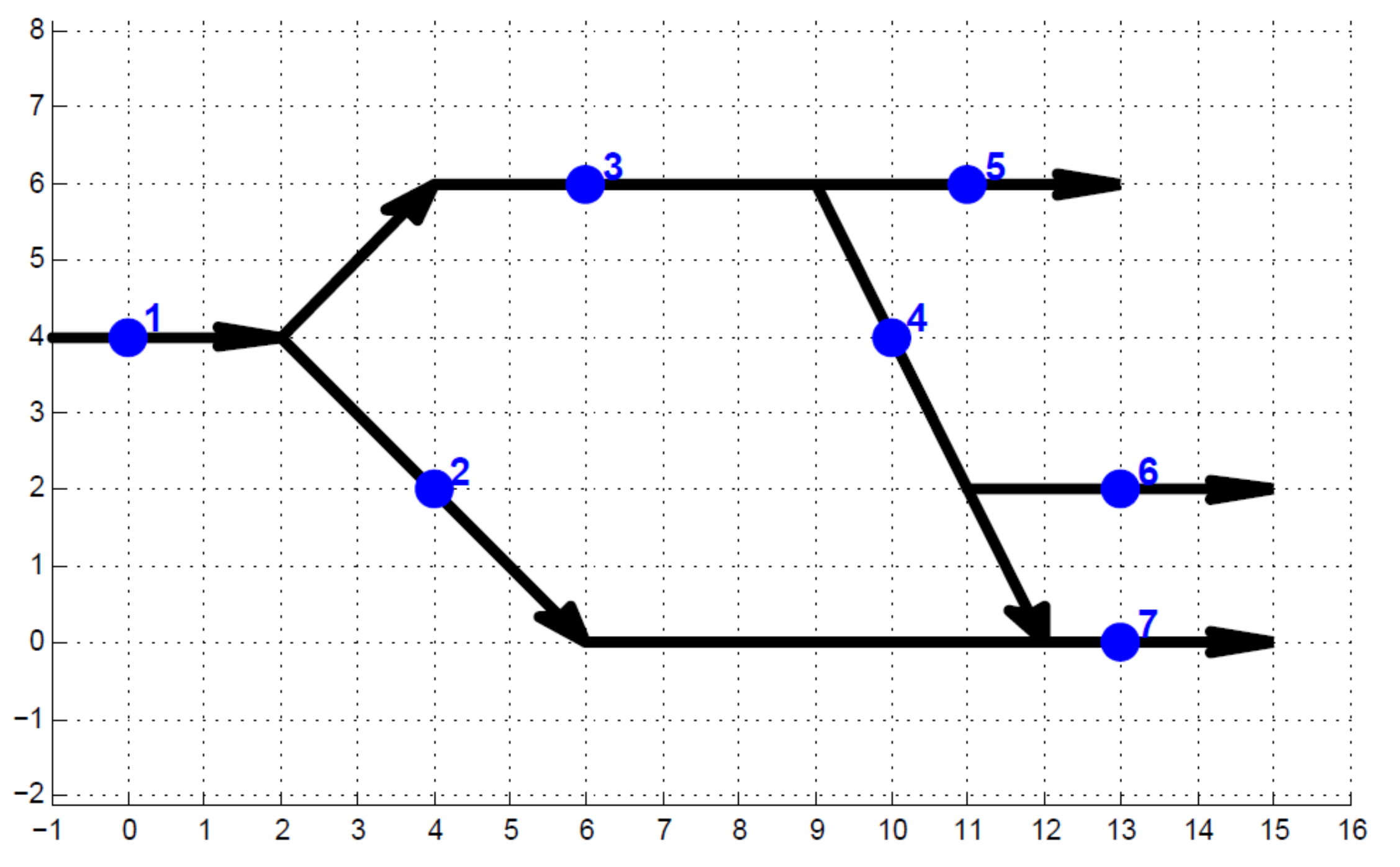}\label{sampleRdGrid}}\\
\subfloat[][Evader Paths showing nodes and time]{\includegraphics[width= \linewidth]{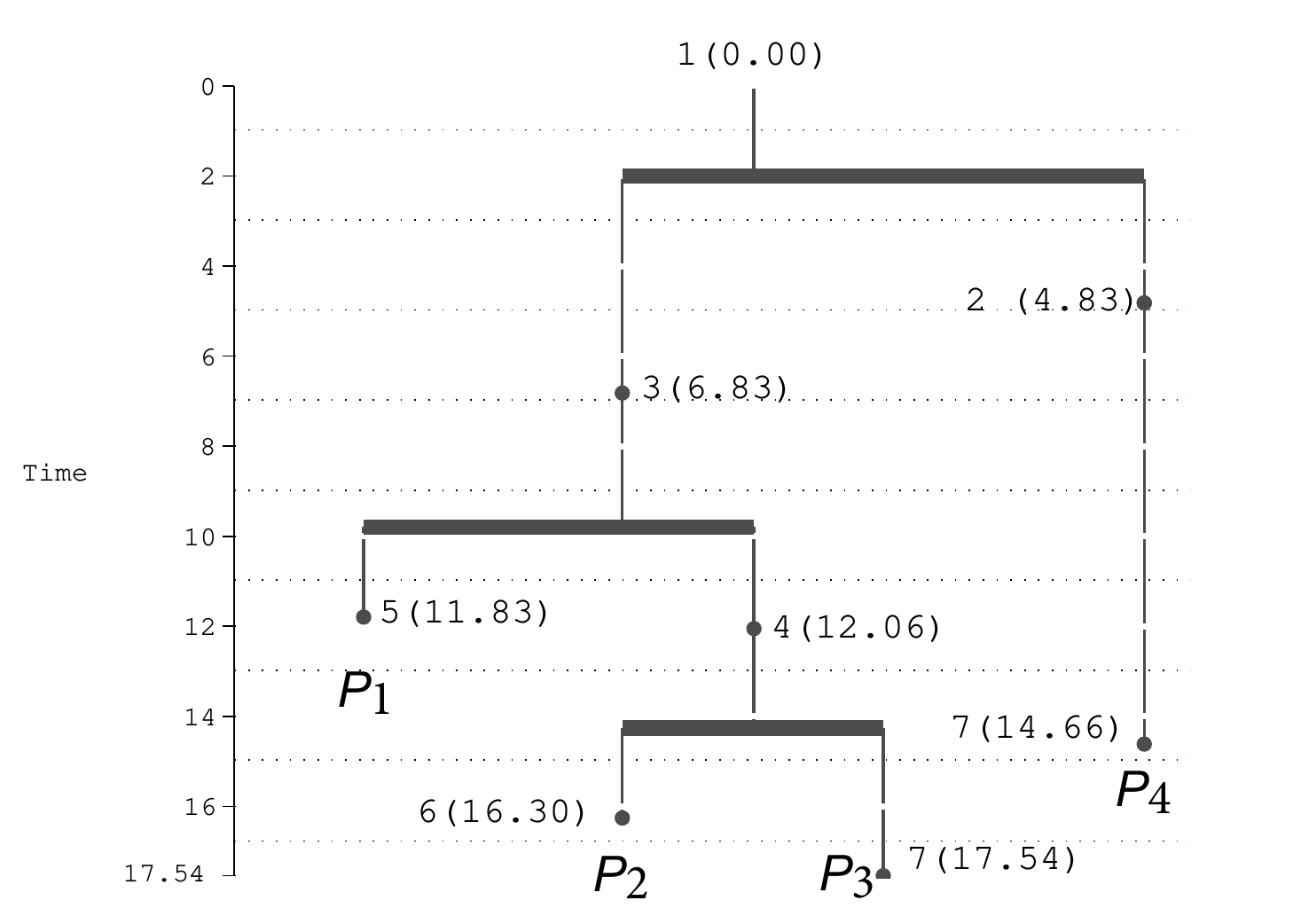}\label{EvaderPaths}}
\end{tabular}
\caption{Example Road Network: a) Grid and b) $4$ Possible Evader Paths} \label{fig:graph-grid}
\end{figure}
we illustrate the simple case of the uncertainty set $\{2,3\}$
and then generalize the method to sets of higher cardinality. Towards this end, we re-draw the example road network in Fig.~\ref{sampleRdGrid}, with a grid in the background, to highlight the $(x,y)$ coordinates of nodes and  the distances along edges. In Fig.~\ref{EvaderPaths}, we show the four different evader paths (ordered from left to right) along with the evader's time of arrival $\mathcal L _j(k)$ (in parentheses) at nodes along each path. Indeed, $P_1=(1\rightarrow 3\rightarrow 5)$, $P_2=(1\rightarrow 3\rightarrow 4\rightarrow 6)$, $P_3=(1\rightarrow 3\rightarrow 4\rightarrow 7)$ and $P_4=(1\rightarrow 2\rightarrow 7)$.

Suppose we wish to compute $\mathcal D(6|\{2,3\})$ having already computed the values, $\mathcal D(6|\{2\})=|P_2|$ and $\mathcal D(6|\{3\})=|P_3|-d_V(6,7)$ from \eqref{eq:bndMaxDelAll}. Here, $|P_2|<|P_3|$ as shown in Fig.~\ref{EvaderPaths}.
To guarantee capture, the pursuer waits at node $6$ until time $|P_2|$. If the evader does not show up at $6$, then the pursuer knows that the evader has taken path $3$ instead. So, the pursuer proceeds to node $7$ and intercepts the evader at  time $|P_3|$. But this is possible only if the following condition is met:
\begin{equation}
\label{eq:go6to7con}
|P_2|+d_V(6,7)<|P_3|.
\end{equation}
Let us suppose that condition \eqref{eq:go6to7con} holds. So, $\mathcal D(6|\{2,3\}) = |P_2|$. 
For all nodes $j\in\{1,\ldots,m\}$ and the uncertainty set $\mathcal I=\{2,3\}$, we have from the RE \eqref{eq:1stepTime}:
\begin{equation}
\label{eq:card2ex2}
\mathcal D(j|\mathcal I)= \max_{u\in\mathcal B} \left[\min\left\{\mathcal D({u}|\mathcal I^r(u)),\mathcal D({u}|\mathcal I^g(u))\right\}
-d_V(j,u)\right],
\end{equation}
where $\mathcal I^r(u) = \mathcal I\cap\mathcal P_u$ and $\mathcal I^g(u)=\mathcal I\backslash\mathcal I^r(u)$ are the updated uncertainty sets for the red and green observations respectively at $u$. We define the restriction $\mathcal B\subset\mathcal U$ as follows. The node $u\in\mathcal B$ if the following conditions are met:\\
1) $\mathcal I^r(u)$ and $\mathcal I^g(u)$ are both singleton sets, and\\
2) $\mathcal D({u}|\mathcal I^g)\geq\min_{k\in \mathcal I^r(u)}\mathcal{L}_u(k)$.\\
Condition 1) implies that $u$ is a special UGS in that, by going to it, the pursuer can reduce the uncertainty set $\{2,3\}$ to either $\{2\}$ or $\{3\}$ depending on whether it observes a red or green UGS (or vice-versa) at $u$. Condition 2) requires that the latest exit time from $u$ for the uncertainty set $\mathcal I^g(u)$ (green observation) must be greater than the earliest possible evader visit time at $u$. Note that this requirement is already satisfied for the red observation (see Lemma~\ref{lem:lowbnd}). In other words, the pursuer will visit $u$ only if there is a possibility that information will become available on whether or not the evader took a path through $u$. Else, there is no value in visiting $u$ and one may as well ignore it. Furthermore, capture is guaranteed from $u$ for either observation, with the corresponding pursuer exit times given by $\mathcal D({u}|\mathcal I^r(u))$ and $\mathcal D({u}|\mathcal I^g(u))$. In the example (see Fig.~\ref{fig:graph-grid}), $\mathcal B=\{6\}$ $\forall j\in \mathcal U$, since visiting node $6$ at time $|P_2|$ can reduce the uncertainty from $\{2,3\}$ to either $\{2\}$ or $\{3\}$ and capture is guaranteed thereafter for either observation. Note that $7\notin\mathcal B$ since it fails condition 2) in that:
\begin{equation*}
\mathcal D(7|\{2\})=|P_2|-d_V(7,6)<|P_2|<|P_3|.
\end{equation*}
So, for the uncertainty set $\{2,3\}$, it makes no sense for the pursuer to go to $7$; since it has to leave $7$ (before time $|P_3|$) with the uncertainty unchanged. Given the triangle inequality constraint \eqref{eq:triangle}, it is therefore sub-optimal for the pursuer to visit node $7$. So, we have justified the restriction $\mathcal B$ in \eqref{eq:card2ex2}, which enables us to compute all the nodes' exit times for the uncertainty set $\{2,3\}$ of cardinality $2$ from the exit times for uncertainty sets of lower cardinality. 

In general, a similar restriction allows us to compute the pursuer exit times for uncertainty sets in an orderly fashion (in the increasing order of cardinality). Let the set of all possible uncertainty sets be $\mathcal Z = 2^{\mathcal I_0}\setminus\emptyset$. We denote the elements of $\mathcal Z$ of cardinality $i$ by $\mathcal I_i^1,\ldots,\mathcal I_i^{o_i}$, where $o_i=\binom{n}{i}$. For instance, $\mathcal I_n^1=\mathcal I_0$. At the other extreme, we have $\mathcal I_1^k=\{k\},\;k=1,\ldots,n$.
Suppose $\mathcal D({j}|\mathcal I_i^q)$ has been computed for $q=1,\ldots,o_i$, $\forall j$ and $i=1,\ldots,r$ for some $r\geq 1$. Then, $\forall q\in\{1,\ldots,o_{r+1}\},\forall j\in\mathcal U,$
\begin{equation}
\label{eq:1stepTime1}
\mathcal D({j}|\mathcal I_{r+1}^q) = \max_{u\in\mathcal B({\mathcal I}_{r+1}^q)} \left[\min_{y\geq -1}\mathcal D({u}|\mathcal I^+(u,y))-d_V(j,u)\right].
\end{equation}
Let $\mathcal I^r(u) = \mathcal I_{r+1}^q\cap\mathcal P_u$ and $\mathcal I^g(u)= \mathcal I_{r+1}^q\backslash\mathcal I^r(u)$. The three distinct possibilities at $u$ are:
\begin{enumerate}
\item[1)] $\mathcal I^r(u) = \mathcal I_{r+1}^q$ which implies that the evader must pass through UGS $u$.
\item[2)] $\mathcal I^r(u) \subset \mathcal I_{r+1}^q$ which implies that the uncertainty is reduced at $u$ for both red and green observations.
\item[3)] $\mathcal I^r(u) = \emptyset$ which implies that a green UGS is the only possible observation at $u$.
\end{enumerate}
We define the restriction $\mathcal B({\mathcal I}_{r+1}^q)\subset \mathcal U$ as follows. A node $u\in\mathcal B({\mathcal I}_{r+1}^q)$ if $\mathcal I^r(u) \subseteq \mathcal I_{r+1}^q$ and the following condition is satisfied:
\begin{equation}
\label{eq:greenCap}
\mathcal D({u}|\mathcal I^g(u))\geq\min_{k\in \mathcal I^r(u)}\mathcal{L}_u(k) \mbox{ if }\mathcal I^r(u) \subset \mathcal I_{r+1}^q.
\end{equation}
Note that the above result already holds, if the observation is a red UGS (see Lemma~\ref{lem:lowbnd}). 

The restriction above implies the following: the pursuer will visit $u$ only if one of two things happen. Either capture if possible at $u$ or the uncertainty is reduced at $u$ with capture guaranteed for either observation (red or green). 
The third possibility 3) implies that the only possible observation at $u$ is a green UGS with no reduction in the uncertainty! Clearly, in this case, there is no information to be gained by visiting $u$ and hence it can be removed from consideration. 

Furthermore, from the triangle inequality constraint \eqref{eq:triangle}, it follows that the only reason to visit $u$ under possibility 1) is to immediately capture the evader at $u$. As before, there is no value in visiting $u$ otherwise, since there is no additional information available at $u$. So, we have the following result.
\begin{lemma}
\label{lem:samecard}
If the optimal control $u$ to \eqref{eq:1stepTime1} is such that  $\mathcal I^r(u) = \mathcal I_{r+1}^q$, then capture occurs at $u$. So, we have $\mathcal D({u}|\mathcal I^r(u))=\min_{k\in\mathcal I^r(u)}\mathcal L_u(k)$. 
\end{lemma}

In conclusion, we note that the uncertainty is either reduced or capture occurs in the next decision epoch. Since the pursuer exit times are already available for uncertainty sets of lower cardinality (former) and it is provided by Lemma~\ref{lem:samecard} for the latter case, we can compute $\mathcal D({j}|\mathcal I_{r+1}^q)$. Finally, to compute $\mathcal D(1|\mathcal I_0)$, we employ the following Ordered Recursive Algorithm (ORA).
\begin{algorithm}{ORA}{\label{sec:algo}}
\qfor $j \qlet 1$ \qto $m$\\
\qfor $k \qlet 1$ \qto $n$\\
 $\mathcal D({j}|\{k\})=|P_k|-d_v(j,s_k^{\ell_k})$\qrof\qrof\\
\qfor $i \qlet 2$ \qto $n-1$\\
\qfor $q \qlet 1$ \qto $o_i$\\
\qfor $j \qlet 1$ \qto $m$\\
 Compute $\mathcal D({j}|\mathcal I_i^q)$ using \eqref{eq:1stepTime1} \qrof\qrof\qrof\\
Compute $\mathcal D(1|\mathcal I_0)$ using \eqref{eq:1stepTime1} \\
 \qreturn $\mathcal D(1|\mathcal I_0)$
\end{algorithm}
Note that the optimal pursuit strategy is constrained to enforce a reduction in entropy! Indeed the entropy i.e., the cardinality of the uncertainty set will reduce, at least by $1$, for every move (including waiting) made by the pursuer. As a result, the game will terminate in no more than $n$ steps/moves! The Algorithm~\ref{sec:algo} has a time complexity of $\mathcal{O}(2^nm\log m)$. This is due to the number of all possible uncertainty sets: $2^n-1$, the number of nodes for which the exit time is computed: $m$ and the time complexity of the $\max$ operation: $\log m$.

\subsection{Pursuer Decision Tree}
\label{sec:dtree}
To evaluate the iterative algorithm prescribed earlier, we implement it on the example problem shown in Fig.~\ref{fig:graph-grid}. We assume that the pursuer travels between any two nodes at a constant speed, $V$. We choose the speed such that \eqref{eq:go6to7con} is satisfied, i.e.,
\begin{eqnarray}
\label{eq:spdChoice}
|P_3|-|P_2|&>&d_V(6,7)=\frac{2}{V}\nonumber\\
\Rightarrow V &>& \frac{2}{\sqrt{5}-1} \approx 1.618,
\end{eqnarray}
where the distance between nodes $6$ and $7$ equals $2$ (see Fig.~\ref{sampleRdGrid}). So, we choose $V=1.62$ and implement Algorithm~\ref{sec:algo}. Fig.~\ref{fig:decTree} shows the decision tree for the pursuer starting with a red UGS at node $1$. The solution dictates that $\mathcal D(1,\{1,2,3,4\})\approx 4.84$ and $\mu(1,\{1,2,3,4\})=3$. Fig.~\ref{fig:decTree} also shows (color coded) the latest pursuer exit times at future nodes visited by the pursuer, for both red and green observations. Eventually, capture of the evader occurs at one of the exit nodes, $5$, $6$ or $7$. Interestingly, the optimal evader path that contributes to the least pursuer exit time at node $1$ is $P_1=(1\rightarrow 3\rightarrow 5)$, which is also the shortest path, i.e., $1=\argmin_k |P_k|$!

If we pick $V=1.61$ instead, we get the decision tree shown in Fig.~\ref{fig:decTree1}. In this case, the maximum delay at node $1$ with a capture guarantee reduces to $\approx 2.9$. This is so because the slower moving pursuer has to capture the evader at node $3$ itself, if the evader picks any path other than $P_4$. In other words, node $3$ acts like an exit node under the reduced speed. If one were to reduce the pursuer speed even further, below some critical speed, $\underline{V}$, the algorithm will return $\mathcal{D}(1,\{1,2,3,4\})=0$, indicating that no initial delay can be tolerated at node $1$ for any speed $V<\underline{V}$. At the other extreme, one can easily confirm that if the pursuer is able to travel at infinite speed, the corresponding $\mathcal{D}(1,\{1,2,3,4\})=|P_1|$, the earliest evader exit time.

\begin{figure}[h]
\begin{center}
\includegraphics[width=\linewidth]{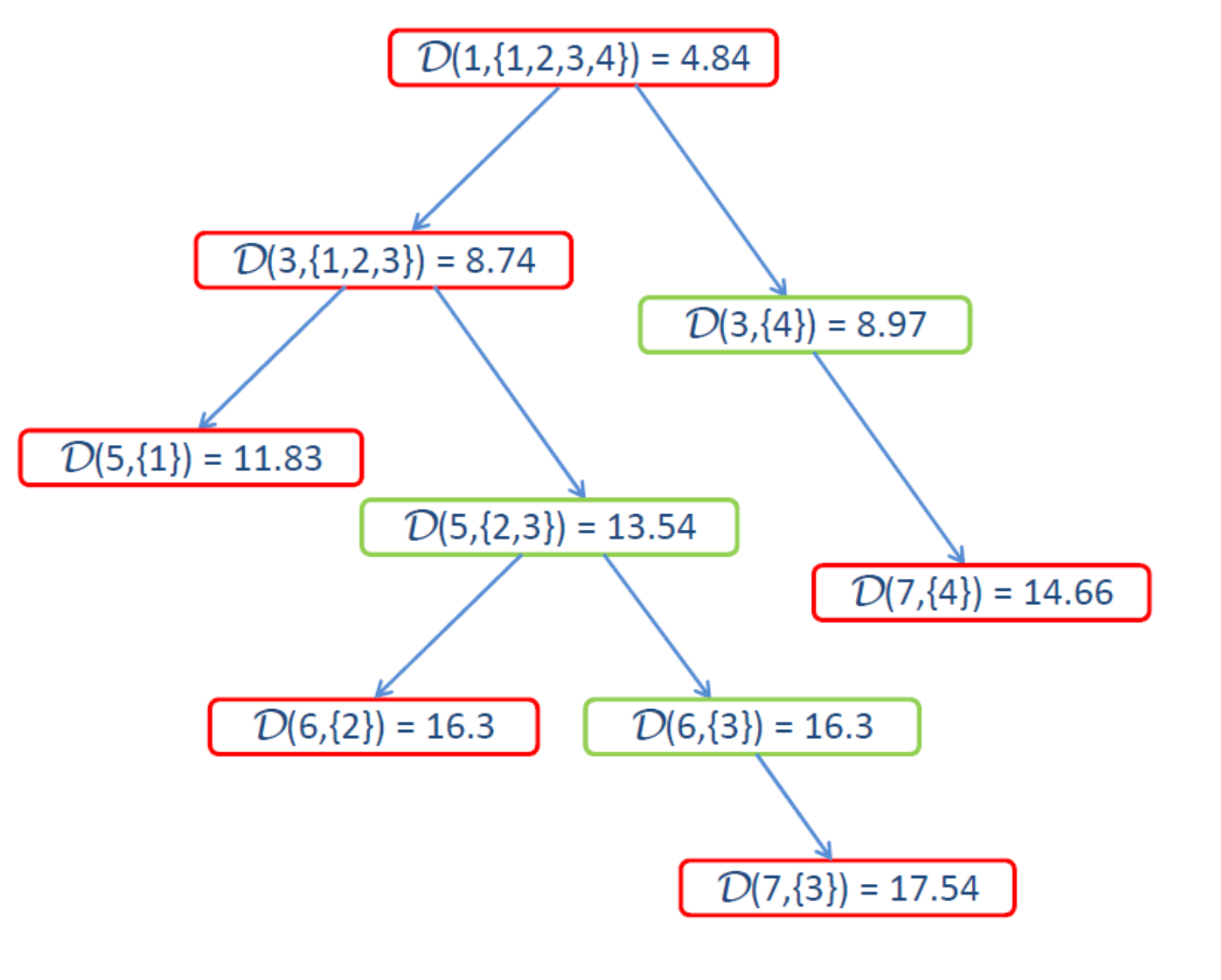}
\caption{Decision Tree and Latest Exit Times for $V=1.62$} \label{fig:decTree}
\end{center}
\end{figure}
\begin{figure}[h]
\begin{center}
\includegraphics[width=0.8\linewidth]{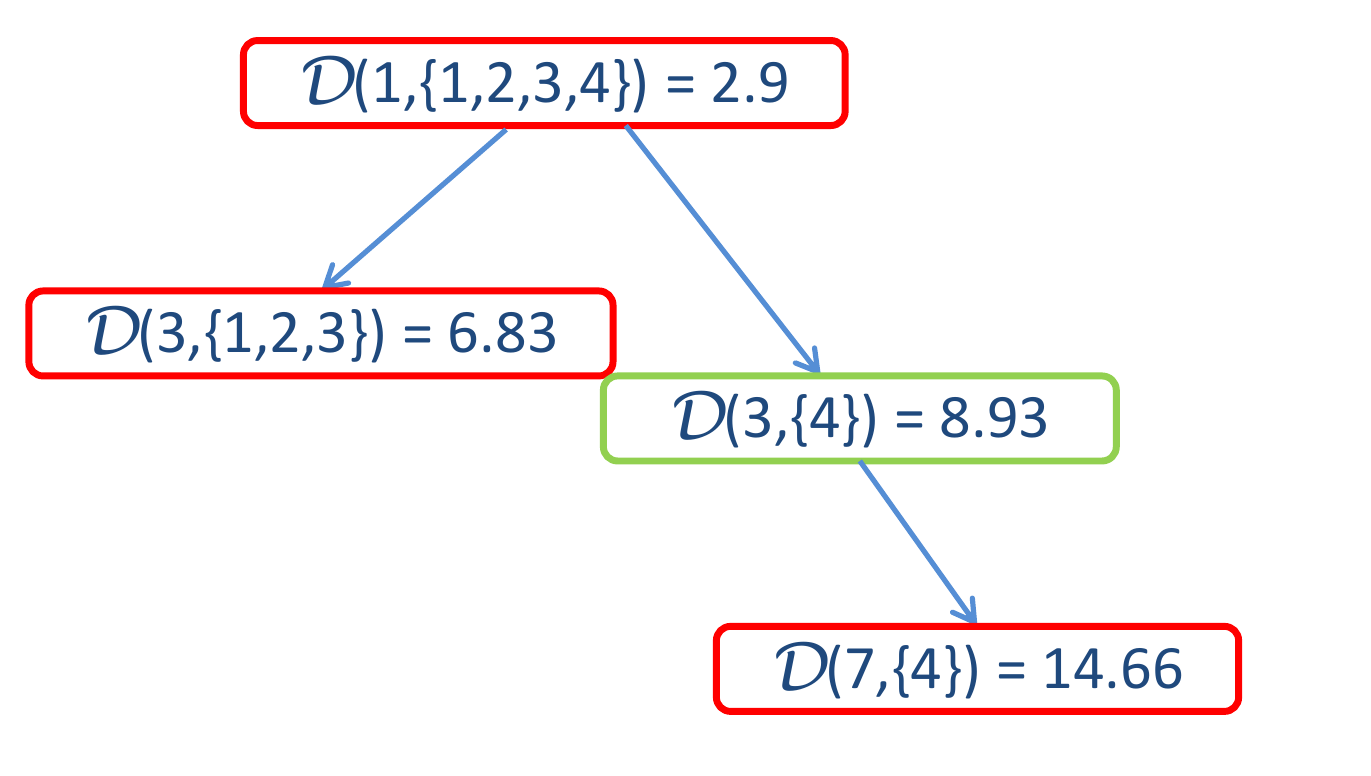}
\caption{Decision Tree and Latest Exit Times for $V=1.61$} \label{fig:decTree1}
\end{center}
\end{figure}

%
\subsection{Reducing the Computational Burden}
\label{sec:campAtt}
Since the algorithm scales exponentially with the number of possible evader paths, we explore avenues that reduce the computation time. We note that for a given graph, $G(\mathcal U, E)$, certain uncertainty sets will never be encountered by the pursuer if it employs a ``guaranteed capture" policy. 
For instance, in the example problem (see Fig.~\ref{EvaderPaths}), the pursuer will never encounter the uncertainty set $\{1,4\}$. The reasoning behind this goes as follows. Initially the pursuer is at node $1$ armed with the uncertainty set $\{1,2,3,4\}$. 
Now the only way the pursuer can reduce the uncertainty set to $\{1,4\}$ is by investigating node $4$ and confirming that paths $2$ and $3$ were indeed not taken. To do so, the pursuer has to (possibly) wait at node $4$ until time $\mathcal T_2(3)\approx 12.06$. But, $\mathcal T_2(3)>|P_1|$ and so, by waiting, the pursuer will necessarily allow the evader to escape via path $1$!
Indeed, it is possible to enumerate all the \emph{realizable} uncertainty sets, that the pursuer will encounter in its search.

For the example problem, the realizable sets listed in Table~\ref{tab:realSets}, are computed in the following manner. At time $0$, the only information available at UGS $1$ is $\{1,2,3,4\}$. At time $\mathcal T_4(2)\approx 4.83$, information is available at UGS $2$ that can reduce the uncertainty to either $\{4\}$ or $\{1,2,3\}$ depending on whether it is red or green. At time $|\mathcal P_1|\approx 11.83$, information is available at UGS $5$ about whether or not the evader took path $1$. Hence the following additional uncertainty sets can be realized: $\{1\},\{2,3,4\}$ and $\{2,3\}$. Note that for any time greater than $|\mathcal P_1|$, $1$ can no longer appear in a uncertainty set, since it would imply that the evader has escaped. This is reflected in the table (see entries after row $4$). We continue the aforementioned procedure to enumerate the sets, until the last UGS/time combination, i.e., $(7,|\mathcal P_3|)$. Upon completing the table, we collect all the sets that appear in column $2$ of Table~\ref{tab:realSets}. This gives us the set of all realizable sets: $\mathcal Y=\{\{1\},\{2\},\{3\},\{4\},\{2,3\},\{1,2,3\},\{2,3,4\},\{1,2,3,4\}\}$.

\begin{table}[ht]
\caption{Realizable Uncertainty Sets at different UGSs in Chronological Order}
\label{tab:realSets}
\centering
        \begin{tabular}{|l|l|}
        \hline
       (UGS, Time) & Realizable Sets \\
        \hline
        $(1,0.00)$ & $\{1,2,3,4\}$\\
        $(2,4.83)$ & $\{1,2,3,4\},\{4\},\{1,2,3\}$\\
        $(3,6.83)$ & $\{1,2,3,4\},\{4\},\{1,2,3\}$\\
        $(5,11.83)$ & $\{1,2,3,4\},\{4\},\{1,2,3\},\{1\},\{2,3,4\},\{2,3\}$\\
        $(4,12.06)$ & $\{4\},\{2,3,4\},\{2,3\}$\\
        $(7,14.66)$ & $\{4\},\{2,3,4\},\{2,3\}$\\
        $(6,16.30)$ & $\{2,3\},\{2\},\{3\}$\\
        $(7,17.54)$ & $\{3\}$\\       
\hline                                
       \end{tabular}
\end{table}
So, we only deal with $8$ sets, as opposed to the $2^4-1=15$ possible combinations. We can now selectively apply Algorithm~\ref{sec:algo}, so that only $\mathcal D(j|\mathcal I),\;\forall\mathcal I\in\mathcal Y$ are computed. Note that there is no loss in optimality, by skipping the non-realizable sets. For a general graph, the reduction in number of sets depends on the structure of the graph. Nonetheless, for large $n$, any reduction from $2^n-1$ could lead to substantial savings in computation time.
\subsection{Partial Information, Dynamic Game, and Dual Control}

We are calculating the maximal allowable delay at UGS $1$ s.t. a pursuit strategy exists which guarantees the evader's capture before the latter reaches one of the 
the goal nodes, $j\in\mathcal G$. This is a deterministic  pursuit-evasion game on a directed acyclic finite graph where the evader's strategy is open-loop control and the pursuer has partial information. Such a game was previously considered in \cite{Krishnamoorthy2013aa,Krishnamoorthy2013ab}, where the highly structured graph considered therein, was a Manhattan grid. Due to the pursuer's information pattern, which is restricted to partial observations of the physical state of the dynamic game, we are running into the difficulties brought about by the \emph{dual control} effect \cite{Basar}, where the current information state determines the pursuer's optimal control while at the same time the information that will become available to the pursuer will be in part determined by his current control. Things are not made easier by the ``minimum time" control flavor of the optimization problem at hand and these difficulties are particularly exacerbated in the context of our dynamic game setting. A solution exists because the optimization problem is discrete and finite but the computational complexity of the algorithm is high.

\section{Conclusions}

The optimal control of a pursuer with limited sensing capability tasked with intercepting a blind evader on a road network instrumented with UGS is considered. The pursuer is interrogating the UGS, some of which were triggered by the passing evader, and as such has access to partial observations only of the physical system's state. Specifically, the maximal allowable delay at an UGS  s.t. a pursuit strategy exists which guarantees the evader's capture before the latter reaches his goal $\mathcal G$ is calculated and the attendant pursuit strategy is obtained. Thus, a deterministic  pursuit-evasion game on a directed acyclic finite graph where the blind evader's strategy is open-loop control and the pursuer has partial information, is solved. Due to the pursuer's information pattern, which is restricted to partial observations of the physical state of the deterministic game at hand, the difficulties brought about by partial information in a dynamic game setting  and the attendant $dual$ $control$ effect, could not be avoided; whence the computational complexity of the solution algorithm. However, in the process of  establishing the maximal delay at UGS $1$ s.t. capture of the evader is possible, the maximal delays for guaranteed capture at all the UGS that are on the $n$ paths emanating from UGS $1$ are also calculated. Finally, the scenario where there are no goal vertices but the directed acyclic graph is infinite is also of interest.

\label{sec:con}

\bibliography{ref_maxdel}

\begin{thebibliography}{1}
\providecommand{\url}[1]{#1}
\csname url@samestyle\endcsname
\providecommand{\newblock}{\relax}
\providecommand{\bibinfo}[2]{#2}
\providecommand{\BIBentrySTDinterwordspacing}{\spaceskip=0pt\relax}
\providecommand{\BIBentryALTinterwordstretchfactor}{4}
\providecommand{\BIBentryALTinterwordspacing}{\spaceskip=\fontdimen2\font plus
\BIBentryALTinterwordstretchfactor\fontdimen3\font minus
  \fontdimen4\font\relax}
\providecommand{\BIBforeignlanguage}[2]{{%
\expandafter\ifx\csname l@#1\endcsname\relax
\typeout{** WARNING: IEEEtran.bst: No hyphenation pattern has been}%
\typeout{** loaded for the language `#1'. Using the pattern for}%
\typeout{** the default language instead.}%
\else
\language=\csname l@#1\endcsname
\fi
#2}}
\providecommand{\BIBdecl}{\relax}
\BIBdecl

\bibitem{Krishnamoorthy2013aa}
K.~Krishnamoorthy, S.~Darbha, P.~Khargonekar, D.~W. Casbeer, P.~Chandler, and
  M.~Pachter, ``Optimal minimax pursuit evasion on a {M}anhattan grid,'' in
  \emph{American Control Conference}, Wasington D.C., 2013, pp. 3427--3434.

\bibitem{Krishnamoorthy2013ab}
K.~Krishnamoorthy, S.~Darbha, P.~Khargonekar, P.~Chandler, and M.~Pachter,
  ``Optimal cooperative pursuit on a {M}anhattan grid,'' in \emph{AIAA
  Guidance, Navigation and Control Conference}, no. AIAA 2013-4633, Boston, MA,
  2013.

\bibitem{Basar}
T.~Ba{\c s}ar, \emph{Control Theory: Twenty-Five Seminal Papers
  (1932-1981)}.\hskip 1em plus 0.5em minus 0.4em\relax Wiley-IEEE Press, 2001,
  ch. Dual Control Theory, pp. 181--196.

\end{thebibliography}
\bibliographystyle{IEEEtran}

\end{document}